\documentclass[11pt]{amsart} 
\usepackage[margin=1.5in]{geometry}

\usepackage{amssymb,upgreek}
\usepackage{amsmath}
\usepackage{url}
\usepackage{bm}
\usepackage{censor}
\usepackage[pagebackref]{hyperref}

\usepackage{enumerate}
\usepackage{xcolor}

\usepackage{tikz}
\usepackage{tikz-cd}

\usepackage{color}
\usepackage{mathtools} 
\usepackage{caption}
\usepackage{subcaption}
\usepackage{fancyhdr}
\usepackage{mathrsfs}
\newtheorem{theorem}{Theorem}[section] 
\newtheorem{lemma}[theorem]{Lemma}
\newtheorem{conjecture}[theorem]{Conjecture}
\newtheorem{proposition}[theorem]{Proposition}
\newtheorem{corollary}[theorem]{Corollary}
\newtheorem*{theorem*}{Theorem}
\newtheorem*{lemma*}{Lemma}
\newtheorem*{corollary*}{Corollary}

\theoremstyle{definition}
\newtheorem{definition}[theorem]{Definition}
\newtheorem{remark}[theorem]{Remark}

\newtheorem{example}[theorem]{Example}
\newtheorem{question}[theorem]{Question}

\newtheorem*{remark*}{Remark}
\newtheorem*{definition*}{Definition}
\newtheorem*{remarks*}{Remarks}
\newtheorem*{addenda*}{Addenda}

\newcommand{\cptwo}{\mathbb{CP}^2}

\newcommand{\LA}{\mathcal L}
\newcommand{\Z}{\mathbb Z}
\newcommand{\C}{\mathbb C}
\newcommand{\A}{\mathbb A}
\newcommand{\Proj}{\textrm{Proj}}
\newcommand{\Aff}{\textrm{Aff}}

\title[Anti-admissible local systems]
{The opposite of admissibility for\\
local systems on line arrangements }

\author[R. Harris]{Robert Harris}
\address{D\'epartement de Math\'ematiques, 
Universit\'e du Qu\'ebec \`a Montr\'eal, 
Montr\'eal, QC, H2X 3Y7, Canada}
\email{harris.robert@uqam.ca}

\date{May 8, 2026.  }



\begin{document}

\begin{abstract}
We study local systems on line arrangements for which no admissibility conditions are satisfied. In doing so we demonstrate the existence of these anti-admissible objects in both the finite and infinite setting and discuss their relation to the combinatorial data of arrangements. We also observe the connection between anti-admissibility and smooth branched coverings and how this connection is related to the branch set being minimal in one sense and maximal in another.
\end{abstract}

\maketitle

\section{Introduction}

Sitting at the core of the study of line arrangements is the question of what can and cannot be determined from the combinatorial data alone. Often this question focuses on the relations between the combinatorics of the arrangement and the topology of their complement. Notably, Rybnikov found that there are arrangements with the same incidence structure which have different fundamental groups \cite{rybnikovfund}. On the other hand, it was known even earlier that the cohomology of the complement (the Orlik-Solomon algebra $H^*(\cptwo\setminus\LA, \C)$) is completely determined by the incidence data \cite{orliksolomon}. This leads us to consider what other differences occur between the settings of homotopy and homology/cohomology. For instance, one might ask how much information is carried by a local system with non-abelian image compared to one with abelian image.   

A local system of rank $n$ on a topological space $Z$ is a homomorphism $\phi:\pi_1(Z)\rightarrow GL(n,\C)$ and historically, there has been a great deal of effort put in to understanding local systems in the setting of $Z$ being the complement of an algebraic curve $\mathcal C\subset Y$. In the context of line arrangements, i.e., $C=\LA=L_1\cup\dots\cup L_n$ and $Y$ being either $\C^2$ or $\cptwo$ (in which case we will call $\phi$ a local system on $\LA$) some important and historically relevant results can be found in the seminal work of Falk \cite{falkcohomology} in which they defined resonance varieties, or in the various works of Libgober and Yuzvinsky \cite{libgoberlocal},\cite{libgoberresults} in which they provided significant elaboration on such objects. This paper however, will instead follow a slightly different route and take inspiration from a notion called admissibility (cf.~\cite{Marco} or Definition~\ref{def: admissible} below) in addition to the well known connection between maps from the fundamental group and branched covers (for the former we refer the reader to \cite{Marco}, \cite{nazirrazaadmissible} and \cite{ngyuenadmissible} whereas for the latter we suggest \cite{hiro93}, \cite{hir83}, \cite{libgobercovers} or \cite{Suciu01}, to name a few).  

Specifically, we strive to further discuss the relation between the combinatorics of line arrangements and algebraic properties of the complement by defining an opposite notion to admissibility, which we call anti-admissibility (Definition~\ref{def: elementary anti-admissible}) in addition to proving various results using this new definition, including. 

\begin{theorem}(a combination of Propositions~\ref{prop: affine integers} and \ref{prop: projective integers} below)

\noindent
      Let $\LA=\{L_1,\dots,L_n\}$ be an arrangement in $Y=\C^2$ (or a non-pencil arrangement in $\cptwo$) for which $H_1(Y\setminus\LA;\Z)$ has rank at least 2. Then $\LA$ is $\Z^k$-anti-admissible for all $2\leq k\leq n$ (or $2\leq k\leq n-1$ in the projective setting).
\end{theorem}

Furthermore, by considering various cases for the combinatorial data of the arrangement, we also provide results for finite groups.

\begin{theorem}(see also Theorem~\ref{theorem: affine mod p} below)
    Let $\LA=\{L_1,\dots,L_n\}$ be a line arrangement in $\C^2$. Then for all primes $p\geq n-1$, $\LA$ is $(\Z/p\Z)^2$-anti-admissible. 
\end{theorem}

Finally, we show how this concept of anti-admissibility fits into the current understanding of coverings branched over line arrangements and we use these observations in conjunction with known correspondences to define anti-admissibility for any local system with finite image (Definition~\ref{def: G anti admissibility}).

\subsection*{Organization} 
In \S\ref{sec:background} we review some standard facts regarding line arrangements, and define anti-admissibility, which is the primary algebraic structure considered in this paper. In \S\ref{sec: affine admissibility}, we provide various results on anti-admissibility, specified to the case when a line arrangement is contained in $\C^2$. In \S\ref{sec: projective admissibility} we show how the results in \S\ref{sec: affine admissibility} behave when the ambient space changes to the complex projective plane. Lastly, in \S\ref{section: branched covers} we review some basic material from \cite{hiro93} on branched coverings before discussing the application of such branched coverings to the study of certain local systems.

\section{Admissibility and anti-admissibility}\label{sec:background}
We begin by fixing some notation and recalling some facts that will be useful throughout this paper.

All homology groups discussed we be assumed to have coefficients in $\Z$. Furthermore, for $x\in\mathbb R$, we will let $\lceil x\rceil$ denote minimum integer which is at least $x$ and $\lceil x\rceil_P$ will specifically denote the minimum \emph{prime} which is at least $x$.

In general, $\LA=\{L_1,\dots,L_n\}$ will denote a line arrangement of $n$\/ complex lines in $Y=\C^2$ (or $\cptwo$) and in the projective setting any arrangement will be assumed to not be a pencil. When viewed in the context of a subset (or subvariety), we will abuse notation and also write $\LA=L_1\cup\dots\cup L_n\subset Y$. 

It is well known that if $\mu_i\in H_1(Y\setminus \LA)$ denotes the (positively oriented) meridian homology class of $L_i\in \LA$ (which may also be denote by $\mu_{L_i}$ at times), then the only relations among the $\mu_i$ are the commutators and possibly $\sum_{i}\mu_i$. Specifically, it is known that
\begin{equation*}
    H_1(\C^2\setminus \LA)\cong
    \langle\mu_1,\mu_2,\dots, \mu_n\mid [\mu_i,\mu_j] \rangle\cong \Z^{n}
\end{equation*}
and
\begin{equation*}
     H_1(\cptwo\setminus \LA)\cong \langle\mu_1,\mu_2,\dots, \mu_n\mid [\mu_i,\mu_j],\mu_1+\mu_2+\cdots+\mu_n \rangle \cong \Z^{n-1} 
\end{equation*} for affine and projective arrangements, respectively (for instance, see ~\cite{Suciu01} or \cite{hiro93}).

The set of points that belong to at least two lines will be denoted by $\textrm{Sing}(\LA)$ and we say an intersection point 
\begin{equation*}
p=p_{i_1,i_2,\ldots,i_r}=L_{i_1}\cap L_{i_2} \cap \cdots \cap L_{i_r} \in\textrm{Sing}(\LA)
\end{equation*}
of $\LA$ has multiplicity $r$ when $r$ is the maximum number of lines in $\LA$ passing through $p$ (we may always assume $i_1\leq\dots\leq i_r$). The value $t_{r}$\/ will refer to the number of intersection points of $\LA$ with multiplicity $r$. Since no two lines can intersect more than once, it is clear that we must have $|\textrm{Sing}(\LA)|\leq n(n+1)/2$.

We now recall what it means for an arrangement to be $k$-admissible.

\begin{definition}\label{def: admissible} 
    {\rm (cf.~Definition~2 in \cite{Marco})}
    For integers $k\geq 2$, a surjection $\phi:H_1(Y\setminus \LA)\rightarrow \Z^{k-1}$ is $k$\emph{-admissible} if 
    \begin{enumerate}
        \item  $\phi(\mu_i)\neq 0$ for all $L_i\in \LA$;
        \item for all $p_{i_1,i_2,\ldots,i_r}\in \textrm{Sing}(\LA)$ the pairs 
    $$\{\phi(\mu_{i_k}),\sum_{j\in \{i_1,i_2,\ldots,i_r\}}\phi(\mu_j)\}$$ 
    are linearly dependent for all $k=1,\dots,r$.
    \end{enumerate}
    Moreover, we say an arrangement $\LA$ is $k$-admissible if it admits such a $k$-admissible surjection.
\end{definition} 

\begin{remark}\label{rem: codomain}
    We note that the codomain in the above definition can be changed to also include elementary abelian $p$-groups of the form $(\Z/p\Z)^k$ (or technically $(\Z/p\Z)^{k-1}$ when mimicking the above). In this context, we will refer to a surjection or arrangement as being $(k,p)$-admissible or $(\Z/p\Z)^{k-1}$-admissible.  
\end{remark}

Even when all meridians have non-trivial image, simply knowing an arrangement fails to be $k$ (or $(k,p)$)-admissible does not tell us how close a surjection can come to satisfying all the required linear dependence conditions.  That being said, restricting to the opposite extreme of Definition~\ref{def: admissible} has geometric significance (cf. Proposition~\ref{prop: main correspondence} below) and along with the above remark regarding the inclusion of $(\Z/p\Z)^k$ as a codomian justifies the following definition. 

\begin{definition}\label{def: elementary anti-admissible} 
    For integers $k\geq 2$ and prime $p\geq 2$, let $G$ be either $\Z^k$ or $(\Z/p\Z)^{k}$. A surjection $\phi:H_1(Y\setminus \LA)\rightarrow G$ is $G$\emph{-anti-admissible} if 
    \begin{enumerate}
        \item  $\phi(\mu_i)\neq 0$ for all $L_i\in \LA$;
        \item for all $p_{i_1,i_2,\ldots,i_r}\in \textrm{Sing}(\LA)$ the pairs 
    $$\{\phi(\mu_{i_k}),\sum_{j\in \{i_1,i_2,\ldots,i_r\}}\phi(\mu_j)\}$$ 
    are linearly \emph{independent} for all $k=1,\dots,r$.
    \end{enumerate}
    We further say an arrangement $\LA$ is $G$-anti-admissible if it admits such a $G$-anti-admissible surjection. When the context is clear we may drop the group and simply write anti-admissible.
\end{definition} 

\begin{remark}
    First, while the above definition has been written to mimic Definition~\ref{def: admissible} as close as possible, we note that condition (1) is effectively superfluous, as it follows from each $\phi(\mu_i)$ being one half of at least one linear independent pair (the only case where this is not true is for an affine arrangement consisting of only parallel lines).  It is also clear that $k\geq 2$ is (almost always) forced by linear independence as well (again, unless we are in this exceptional case of only parallel lines in $\C^2$). For these reasons, we will always assume that an arrangement $\LA$ consists of at least two lines which are not parallel so that $\textrm{Sing}(\LA)$ is non-empty (and at least three lines in the projective case so that the set of all surjections $\Z^{n-1}\rightarrow G$ is non-empty). 
    
    Due to the significance of property (2), we will sometimes refer to it as the anti-admissibility condition (typically at a given intersection point).

\end{remark}

\begin{example}\label{example: canonical}
(cf.~page 968 in \cite{kul04})
Let $\LA=\{L_1,\dots,L_n\}$\/ be any (non-pencil) line arrangement consisting of $n>2$\/ complex lines in $\cptwo$ and fix $G=\Z^{n-1}$ or $G=(\Z/p\Z)^{n-1}$ for some prime $p$. Consider the surjective homomorphism 
$$\phi : H_1(\cptwo \setminus\LA)\cong\Z^{n-1}\rightarrow G$$
given by 
\begin{equation*}\label{eq: canonical surjection}
    \phi(\mu_i)=
    \begin{cases}
        \mathbf{e}_i \quad &i\in \{1,\dots,n-1\},\\
        -\sum_{i<n}\mathbf{e}_i \quad &i=n,
    \end{cases}
\end{equation*}
where $\mathbf{e}_i$ is the $i$th standard basis vector of $G$.  
It is easy to see that for any possible intersection point $p_{i_1,i_2,\ldots,i_r}\in \textrm{Sing}(\LA)$, $\phi(\mu_{i_k})$ and $\sum_{j\in \{i_1,i_2,\ldots,i_r\}}\phi(\mu_j)$ will be linearly independent. Consequently, any arrangement with $n$ lines is both $\Z^{n-1}$ and $(\Z/p\Z)^{n-1}$-anti-admissible (for any prime p). 

A similar analysis also shows that any arrangement of $n\geq 2$ lines in $\C^2$ will be both $\Z^{n}$ and $(\Z/p\Z)^{n}$-anti-admissible (for any prime p).   
\end{example}

\begin{remark}\label{rem: small groups interesting}
By comparing the size of an arrangement to the rank of the abelian group $G$, we see that Example~\ref{example: canonical} contains the `largest' groups for which a fixed arrangement can be anti-admissible. Namely, if $k>n-1$ (or $k>n$ and $Y=\C^2$), then no surjection exists. Furthermore, the example shows that these extremal cases can be achieved for any arrangement and so in general the most interesting cases of $G$-anti-admissibility occur when one attempts to minimize the size of $G$. 
\end{remark}

We end this section with some observations which will be of use throughout the remainder of this paper.

\begin{remark}\label{rem: double points LI}
    We highlight that the anti-admissibility condition at a double point $p=L_i\cap L_j$ of an arrangement is equivalent to requiring that $\phi(\mu_i)$ and $\phi(\mu_j)$ be linearly independent. Importantly, this implies that the existence and the number of generic lines (lines which intersect each other line at a double point) in an arrangement places a restriction on what groups can appear (cf. Example~\ref{example: general position k=2} below).
    
    On the other hand, it is possible for the sum of $r\geq 3$ pairwise linearly independent elements to not be independent from one of them (consider $(1,1),(1,2)$ and $ (1,3)$ in $(\Z/7\Z)^2$), and so for arrangements containing points of higher multiplicity, other interesting restrictions may arise, albeit often on a case by case basis (for example,~Lemma~\ref{lemma: triple points bad}). 
\end{remark}

\begin{example}\label{example: general position k=2}
   
    As it will show up at various times in the following sections, we will let $\mathcal G_n$ specifically denote a line arrangement consisting of $n$ lines in general position. For a fixed prime $p$, suppose there exists a surjective homomorphism 
    $$\phi : H_1(Y \setminus \mathcal G_n)\rightarrow(\Z/p\Z)^{2}$$
    which is $(\Z/p\Z)^{2}$-anti-admissible. As per the above remark, this implies that any two elements from $\{\phi(\mu_i)\mid L_i\in \mathcal G_n\}$ must be linearly independent.

    Now, in $(\Z/p\Z)^2$ (and more generally $(\Z/p\Z)^k$), two (non-zero) elements are linearly independent if and only if they generate distinct order $p$ subgroups. That is, non-zero $a,b\in (\Z/p\Z)^2$ are linearly independent if and only if $\langle a\rangle\cap \langle b\rangle=\{0\}$.    
    
    From this observation and by counting the number of distinct order $p$ subgroups of $(\Z/p\Z)^{2}$, we conclude that when $p+1<n$, there are simply not enough elements (or more accurately, subgroups) to assign to each $\mu_i$. That is, $\mathcal G_n$ is never $(\Z/p\Z)^{2}$-anti-admissible for $p<n-1$.  

    We will return to/and utilize this equivalence between linear independence and subgroup intersections many times below (with some notable cases being Theorem~\ref{theorem: affine mod p} and Proposition~\ref{prop: main correspondence}).
\end{example}

\section{Anti-admissibility for affine arrangements}\label{sec: affine admissibility}

In this section, we will fix $Y=\C^2$ unless otherwise stated. As noted above, the only relations between meridians around lines in $H_1(\C^2\setminus \LA)$ are the commutators. Also, unlike in projective arrangements, two distinct lines need not intersect and so $\textrm{Sing}(\LA)$ can often be significantly simpler (or at least smaller). Consequently, anti-admissibility is generally easier to prove in this setting.

We start by further justifying the discussion of Remark~\ref{rem: small groups interesting} with two results.

\begin{proposition}\label{prop: affine integers}
    Let $\LA=\{L_1,\dots,L_n\}$ be an arrangement in $\C^2$. Then $\LA$ is $\Z^k$-anti-admissible for all $2\leq k\leq n$.
\end{proposition}

\begin{proof}

    We start by considering $k=2$ and we will define $\phi:H_1(Y\setminus \LA)\rightarrow \Z^2$ inductively, starting with $\phi(\mu_1)=(1,q_1)=(1,2)$ and $\phi(\mu_2)=(1,q_2)=(1,3)$ to guarantee surjectivity. Now, for $i+1> 2$, we set $\phi(\mu_{i+1})=(1,q_{i+1})$ for a prime $q_{i+1}$ chosen in the following way:
    
    Let $\{i_1,\dots,i_{r-1}\}\subset \{1,\dots,i\}$ be any non-empty subset and consider the equation
    $$\frac{1}{r}\left(q+\sum_{m=1}^{r-1}q_{i_m}\right)=q,$$
    and the $r-1$ many equations of the form
    $$\frac{1}{r}\left(q+\sum_{m=1}^{r-1}q_{i_m}\right)=q_{i_l},$$
    for $i_l\in \{i_1,\dots,i_{r-1}\}$. Then each of these $r$ equations has a unique rational solution. Let $A_{i+1}$ denote the (finite) set defined as the union of all such solutions, ranging over all subsets of $ \{1,\dots,i\}$. We choose $q_{i+1}$ to be the smallest prime not contained in $A_{i+1}$.
    Repeating this procedure we can define $\phi(\mu_i)=(1,q_i)$ for all $1\leq i\leq n$.

    Now suppose $p_{i_1,i_2,\ldots,i_r}\in \textrm{Sing}(\LA)$ is any intersection point (with $i_r\neq 2$), then we consider the pairs
    $$\{(1,q_{i_l}),\sum_{m=1}^r(1,q_{i_m})\}.$$
    If for some $l$ this pair is not linearly independent, then there are non-zero $a_1,a_2\in \Z$ such that
    \begin{align*}
        a_1(1,q_{i_l})=a_2(r,\sum_{m=1}^rq_{i_m}).
    \end{align*}
    It then follows that $q_{i_l}=\frac{1}{r}\sum_{m=1}^rq_{i_m}$ and so $q_{i_r}\in A_{i_r}$, a contradiction. If $i_r=2$, this corresponds to $L_1\cap L_2$, but it is clear that $(1,2)$ and $(1,3)$ are linearly independent to $(2,5)$. Therefore all pairs above will be linearly independent and so $\phi$ is $\Z^2$-anti-admissible.

    For $k\geq 3$, we modify the above construction by setting $\phi(\mu_{n-j})=e_{k-j}$ for $j=0,1,\dots,k-3$ so that
    \begin{equation*}
    \phi(\mu_i)=
    \begin{cases}
        (1,q_i,0,\dots,0) \quad &i\in \{1,2,\dots,n-k+2\},\\
         \mathbf{e}_i \quad &i\in\{n-k+3,\dots,n\}.
    \end{cases}
\end{equation*}
One can check that for any $r$-fold intersection point, consisting of $m$ many lines with index in $\{n-k+3,\dots,n\}$ and $r-m$ many in $\{1,2,\dots,n-k+2\}$, all required pairings will still be linearly independent. Furthermore, the set 
$$\{\phi(\mu_1),\phi(\mu_2),\phi(\mu_{n-k+3}),\dots,\phi(\mu_n)\}$$
generates $\Z^k$ when $k\leq n$. 
\end{proof}

\begin{remark}
    Strictly speaking, the choices for the $q_i$ in the above proof did not need to be restricted to primes and so as $n$ grows, one can still obtain a sequence of integers. Nonetheless, the case in the proof will generate a proper subsequence of primes ($2,3,5,11,13,\dots$). To our knowledge, the question as to whether there are any notable properties of this sequence or a sequence defined in a similar manner remains to be shown, although we must admit this area is out of our expertise. Regardless, we would be interested in learning the answer. 
\end{remark}

\begin{remark}\label{rem: exhaustively check}
    The inherit finiteness of the combinatorics of line arrangements implies that the two results below (and other similar statements) can (at least in principle) be proved by exhaustively checking if any possible assignment of images to meridians will satisfy the anti-admissibility condition (an idea echoed by Example~\ref{ex: differing spectra ceva}). However, we provide arguments which fall between the abstract and the brute force approach as similar arguments will appear at other times throughout this paper.  
\end{remark}

In contrast to the integral case, when we instead consider anti-admissibility for groups of the form $G=(\Z/p\Z)^k$, we see that the analogue of the above proposition is simply not true. For instance we have the following (which is also true for projective arrangements).

\begin{lemma}\label{lemma: triple points bad}
    Suppose $\LA$ is an arrangement consisting of at least one point $p\in\textrm{Sing}(\LA)$ with odd intersection multiplicity, then $\LA$ is not $(\Z/2\Z)^2$-anti-admissible.
\end{lemma}

\begin{proof}
    Suppose $\phi:H_1(Y\setminus \LA)\rightarrow (\Z/2\Z)^2$ was anti-admissible and (up to relabeling) let $p=L_1\cap\dots\cap L_{2k+1}$. Since $\sum_{i=1}^{2k+1}\phi(\mu_i)$ must be non-trivial, we must have that $\{\phi(\mu_1),\dots,\phi(\mu_{2k+1})\}$ consists of only two elements $g_1,g_2\in (\Z/2\Z)^2$. Furthermore, since the multiplicity of $p$ is odd, then without loss of generality, $g_1$ must show up in $\phi(\mu_1),\dots,\phi(\mu_{2k+1})$ exactly an even number of times. It then follows that $\sum_{i=1}^{2k+1}\phi(\mu_i)=g_2$, which contradicts anti-admissibility.   
\end{proof}

In addition to the above statement, triple points imply additional restrictions on anti-admissible morphisms (which also holds for the projective case).

\begin{lemma}\label{lemma: triple points p=3}
    Suppose $\phi:H_1(Y\setminus \LA;\Z)\rightarrow (\Z/3\Z)^2$ is anti-admissible, and $p=L_{1}\cap L_{2}\cap L_{3}$ is a triple point. Then $\phi(\mu_{i})=\phi(\mu_{j})$ for some $1\leq i<j\leq 3$.
\end{lemma}

\begin{proof}
    We note that we cannot have $\langle\phi(\mu_1)\rangle=\langle\phi(\mu_2)\rangle=\langle\phi(\mu_3)\rangle,$ as then all pairings $\{\phi(\mu_i),\phi(\mu_1)+\phi(\mu_2)+\phi(\mu_3)\}$ will be linearly dependent. Let us assume that each $\phi(\mu_i)$ generates a distinct (non-trivial) subgroup, which we will label $H_i$ respectively. Since $(\Z/3\Z)^2$ contains exactly four distinct subgroups of order $p=3$, we will let $H_4$ denote the remaining subgroup. Now, $\phi(\mu_1)+\phi(\mu_2)\not\in H_1,H_2$, as in either other case it would imply that $\phi(\mu_1)$ and $\phi(\mu_2)$ are linearly dependent and in particular $H_1=H_2$, a contradiction. Furthermore, if $\phi(\mu_1)+\phi(\mu_2)\in H_3$, then 
    $\{\phi(\mu_3),\phi(\mu_1)+\phi(\mu_2)+\phi(\mu_3)\}$ would be linearly dependent, contradicting anti-admissibility. Hence $\phi(\mu_1)+\phi(\mu_2)\in H_4$. However, since $\phi(\mu_3)\not\in H_4$, we then have $\phi(\mu_1)+\phi(\mu_2)+\phi(\mu_3)\not\in H_4$, which also leads to one of the linear independence conditions failing. 

    We can therefore conclude that not all generated subgroups are distinct, and so without loss of generality, suppose $\langle\phi(\mu_1)\rangle=\langle\phi(\mu_2)\rangle.$ Now, if $\phi(\mu_1)\neq \phi(\mu_2)$, then $\phi(\mu_1)+\phi(\mu_2)+\phi(\mu_3)=\phi(\mu_3)$, which again contradicts anti-admissibility. The result follows.
\end{proof}

While in general, $(\Z/p\Z)^k$-anti-admissibility is not guaranteed for any $p$ and $k$, we may still utilize the extension idea from Proposition~\ref{prop: affine integers} to induct our way up.

\begin{lemma}\label{lemma: vertical extension affine}
    Suppose a line arrangement $\LA=\{L_1,\dots,L_n\}$ in $\C^2$ is $(\Z/p\Z)^k$-anti-admissible for a fixed prime $p$. If $k+1\leq n$, then $\LA$ is $(\Z/p\Z)^{k+1}$-anti-admissible.
\end{lemma}

\begin{proof}
    Suppose $\phi : H_1(\C^2 \setminus\LA)\rightarrow G=(\Z/p\Z)^k$ is the given anti-admissible surjection. Then since $G$ is a $\Z/p\Z$ vector space, $\{\phi(\mu_1),\dots,\phi(\mu_n)\}$ contain a basis for $G$. We relabel the lines in $\LA$ so that $\{\phi(\mu_1),\dots,\phi(\mu_k)\}$ is such a basis. Now we extend $\phi$ to a homomorphism $\phi' : H_1(\C^2 \setminus\LA)\rightarrow (\Z/p\Z)^{k}\oplus(\Z/p\Z)$ via
    $$\phi'(\mu_i)=(\phi(\mu_i),\delta_{i,k+1}).$$
    Now, $\phi'$ is well defined since $H_1(\C^2\setminus\LA)$ is freely generated by the $\mu_i$'s.  Moreover, since $\{\phi(\mu_1),\dots,\phi(\mu_k)\}$ is a basis for $G$, there are $a_i\in\{0,1,\dots, p-1\}$ not all zero for which
    $$\sum_{i=1}^ka_i\phi(\mu_i)=\phi(\mu_{k+1}).$$
    It follows that 
    $$\phi'(\mu_{k+1}-\sum_{i=1}^ka_i\mu_i)=(\phi(\mu_{k+1})-\sum_{i=1}^ka_i\phi(\mu_i),1)=(0,\dots,0,1).$$
    and so $G\oplus \Z/p\Z$ is spanned by $\{\phi'(\mu_1),\dots,\phi'(\mu_{k+1})\}$, provided $k+1\leq n$. 
    
    Finally, suppose $p_{i_1,i_2,\ldots,i_r}\in \textrm{Sing}(\LA)$ is any intersection point so that 
    \begin{align*}
           \sum_{m=1}^{r}\phi'(\mu_{i_m})&=(\sum_{m=1}^{r}\phi(\mu_{i_m}),\sum_{m=1}^{r}\delta_{i_m,k+1})\\
           &=
           \begin{cases}
               \left(\sum_{m=1}^{r}\phi(\mu_{i_m}),1\right) \quad p_{i_1,i_2,\ldots,i_r}\in L_{k+1}\\
               \left(\sum_{m=1}^{r}\phi(\mu_{i_m}),0\right) \quad otherwise.\\
           \end{cases}
    \end{align*}
    
    If $p_{i_1,i_2,\ldots,i_r}\not\in L_{k+1}$, then linear independence of $\phi'(\mu_{i_l})$ and $\left(\sum_{m=1}^{r}\phi(\mu_{i_m}),0\right)$ follows immediately from the anti-admissibility of $\phi$. 
    
    On the other hand, if $a_{i_l}\phi'(\mu_{i_l})=b_{i_l}\left(\sum_{m=1}^{r}\phi(\mu_{i_m}),1\right)$, then $b_{i_l}=0$ when $i_l\neq k+1$ (which implies $a_{i_l}=0$) or $b_{k+1}=a_{k+1}$. The latter of these implies $\phi(\mu_{k+1})=\sum_{m=1}^{r}\phi(\mu_{i_m})$, which would contradict the anti-admissibility of $\phi$. Consequently, in all cases, $a_{i_l}=b_{i_l}=0$ and so linear independence follows. Hence, $\phi'$ is $(\Z/p\Z)^{k+1}$-anti-admissible.
\end{proof}

As a consequence of the above lemma, we see that we can provide an analogue for Proposition~\ref{prop: affine integers} in the case of $(\Z/p\Z)^k$ by providing a lower bound for the prime for which all arrangements are $(\Z/p\Z)^2$-anti-admissible. 

\begin{theorem}\label{theorem: affine mod p}
    Let $\LA=\{L_1,\dots,L_n\}$ be a line arrangement in $\C^2$. Then for all primes $p\geq n-1$, $\LA$ is $(\Z/p\Z)^2$-anti-admissible. 
\end{theorem}

\begin{proof}

Let $\LA$ be a fixed arrangement with $n$ lines. Similar to Proposition~\ref{prop: affine integers}, we will inductively define a homomorphism $\phi:H_1(Y\setminus \LA)\rightarrow (\Z/p\Z)^2$. However, we will first prove the case $p>n$ and then explain the required modifications to to justify the inequality $p\geq n-1$ afterwards. 

First, we note that $(\Z/p\Z)^2$ contains $p+1$ distinct subgroups of order $p$, for instance, those generated by and $(1,m)$ for $m=0,\dots,p-1$ in addition to one generated by $(0,1)$.  Recalling the observation in Example~\ref{example: general position k=2}, after enumerating these subgroups (and their generators, respectively) by $H_m=\langle g_m\rangle =\langle(1,m)\rangle$ for $m\in\{1,\dots,p\}$ and $H_{p+1}=\langle g_{p+1}\rangle=\langle(0,1)\rangle$, we wish to associate to each $\mu_i$ a subgroup $H_{m_i}$ by defining $\phi(\mu_i)=c_ig_{m_i}$ for some $c_i\in\{1,\dots,p-1\}$ and some $m_i\in\{1,\dots,p+1\}$. 

We start by setting $\phi(\mu_1)= g_1=(1,1)$ and $\phi(\mu_2)=g_2=(1,2)$, so that $\phi(\mu_1),\phi(\mu_2)$ and $\phi(\mu_1)+\phi(\mu_2)$ are nontrivial. Then for $i\geq 3$, let $\LA_i=\{L_1,\dots, L_i\}$ denote the sub-arrangement built with by only considering the first $i$ many lines in $\LA$. In this case, the set of intersection points of this sub-arrangement which lay on $L_i$, denoted $Q_i:=\textrm{Sing}(\LA_i)\cap L_i=\{q_1,\dots,q_{i_l}\}$, will satisfy $l\leq i-1<n<p+1$. For each $q\in Q_i$, we compute the (non-trivial) subgroup 
\begin{equation*}
    H_q=\langle
    \sum_{L\in \LA_{i-1},\  
    q\in L }\phi(\mu_L)\rangle.
\end{equation*}
Since $l<p+1$, there is at least one order $p$ subgroup $H<(\Z/p\Z)^2$ for which $H\not\in \{H_q\mid q\in Q_i\}$. Choose $m_i\in\{1,\dots,p+1\}$ so that $H_{m_i}=H$. 

Next, for each fixed $q\in Q_i$, we compute the elements 
\begin{equation*}
    h_{c}^q:=cg_{m_i} +\sum_{L\in \LA_{i-1},\  
    q\in L }\phi(\mu_L)
\end{equation*}
for $c\in \{1,\dots,p-1\}$ and notice that each $h_{c}^q$ generates a distinct order $p$ subgroup of $(\Z/p\Z)^2$. Indeed, we first note that $h_c^q$ is non-zero since otherwise the (non-zero) element $cg_{m_i}$ (and hence also $g_{m_i}$) would generate the same subgroup as $\sum_{L\in \LA_{i-1},\  
q\in L }\phi(\mu_L)$, contradicting how $m_i$ was chosen. Moreover, if distinct $c,c'\in\{1,\dots,p-1\}$ satisfy $\langle h_{c}^q\rangle=\langle h_{c'}^q\rangle$, then 
\begin{equation*}
    cg_{m_i} +\sum_{L\in \LA_{i-1},\  
    q\in L }\phi(\mu_L)=d\left(c'g_{m_i} +\sum_{L\in \LA_{i-1},\  
    q\in L }\phi(\mu_L)\right)
\end{equation*}
for some $d\in \{1,\dots,p-1\}$ and in particular, 
\begin{equation*}
    (c-dc')g_{m_i}=(d-1)\sum_{L\in \LA_{i-1},\  
    q\in L }\phi(\mu_L).
\end{equation*}
Similar to above, $H_{m_i}$ and $\langle \sum_{L\in \LA_{i-1},\ q\in L }\phi(\mu_L)\rangle$ intersect trivially which implies $d-1\equiv c-dc' \equiv 0\pmod{p}$ and consequently, $c=c'$. 

For each $q\in Q$, we consider the values of $c$ for which $h_c^q$ generates the same subgroup of $\phi(\mu_L)$ for some line $L\in \LA_{i-1}$ passing through $q$ and we denote this set of all such values by $C_q$. Then $C_q$ has at most $r_q-1$ many elements where $r_q$ is the multiplicity of $q$ in $\LA_i$. It then follows that the set $C_{q_1}\cup\dots\cup C_{q_{i_l}}$ has at most $i-1$ elements. In particular, since $i-1\leq n-1<p-1$, we can find $c\in \{1,\dots,p-1\}\setminus (C_{q_1}\cup\dots\cup C_{q_{i_l}})$, which we will denote by $c_i$.

We claim that with these choices, $\phi(\mu_i)=c_ig_{m_i}$ (recalling $c_1=c_2=1,$ and $ m_1=1,m_2=2$) demonstrates $(\Z/p\Z)^2$-anti-admissible. First, $\phi$ is surjective since $g_1$ and $g_2$ generate $(\Z/p\Z)^2$. Now, if $q=q_{i_1,i_2,\ldots,i_r}$ is any intersection point, and $i$ is the largest index of all lines passing through $q$. Then we notice that that the sub-arrangement $\LA_i$ contains all lines passing through $q$ and so
$$ \sum_{j=1}^{r}\phi(\mu_{i_j})=h_{c_i}^q=c_ig_{m_i} +\sum_{L\in \LA_{i-1},\  q\in L }\phi(\mu_L).$$
Considering linear independence, we note that $\phi(\mu_{i_j})$ and $h_{c_i}^q$ will generate distinct subgroups when $i_j\neq i$ by our choice of $c_i$. When $i_j=i$, if linear independence is not satisfied then $dg_{m_i}=h_{c_i}^q$ for some $d\in \Z/p\Z$ and thus
$$(d-c_i)g_{m_i}=\sum_{L\in \LA_{i-1},\  q\in L }\phi(\mu_L).$$
In this case, we would be contradicting the choice of $m_i$. It follows that $\phi$ is $(\Z/p\Z)^2$-anti-admissible.

Now, to improve this result to $p\geq n-1$, we note that the obstruction is that when we attempt to pick the coefficients $c_{n-1}$ and $c_{n}$ for the meridians $\mu_{n-1}$ and $\mu_{n}$, respectively, the set $\{1,\dots,p-1\}\setminus (C_{q_1}\cup\dots\cup C_{q_{i_l}})$ may be empty. Consequently, we must show that for any arrangement, we may order the lines so that for $L_{n-1}$ and $L_n$, we have $|(C_{q_1}\cup\dots\cup C_{q_{i_l}})|\leq n-3$ (and hence less than $p-1$) in the corresponding step above. From here on, we will somewhat abuse notation and write $C_{i}=(C_{q_1}\cup\dots\cup C_{q_{i_l}})$ when we refer to this set, as the specific labeling of points will not be vital to the arguments.

Also, one can easily check that for any arrangement with $n=2,3$ or $4$, the theorem is true. Therefore, we will assume that $n\geq 5$.

Note that if $q\in L_{i}$ is a double point of $\LA_i$ (say $q=L_i\cap L_j$ for some $j<i$), then since we can choose $H_{m_i}$ to be any order $p$ subgroup distinct from $H_{m_j}$, then $C_q=\emptyset$. On the other hand, if $L_j$ does not intersect $L_i$, then there is also no restriction on the choice of $c_i$ coming from $\phi(\mu_j)$ (as there is no restriction on the image at all). Therefore, if there are at least two lines in $\LA$ which each contain (a): two double points, (b): two ``missing intersection points" due to parallel lines (also in $\LA$) or (c): at least one of each type, then both lines will necessarily have $|C_i|\leq n-3$, and so relabeling $\LA$ so that these are indexed by $n-1$ and $n$, we have our desired result. In fact, regarding $L_{n-1}$, we only need it to contain a single double point (as an intersection in $\LA_{n-1}$) or a single other parallel line (that is not $L_n$). Consequently, we see that (b) occurring for a single line is sufficient, and so we ignore this case entirely.

Now, let us assume instead that there is only one line $L\in\LA$ satisfying (a) or (c). Similar to above, it is clear that setting $L_n=L$ will satisfy $|C_{n}|\leq n-3< p-1$. We will now provide the argument for finding $L_{n-1}$ in the case that $L$ contains two double points, the other case is similar. First, re-index so that the two lines which intersect $L$ at double points become $L_1$ and $L_2$. Then these two lines themselves must also intersect at a point with multiplicity $r\geq 3$ (as otherwise we would be in a previous case), and so we take one of the other $r-2$ many lines passing through $L_1\cap L_2$ and label it $L_{n-1}$ (Figure~\ref{fig: two double}). 

\begin{figure}[h]
    \begin{center}
    \includegraphics[scale=.5]{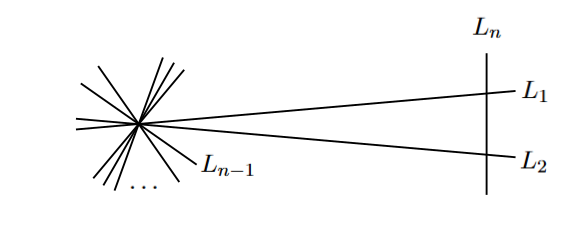}
    \caption{Picking $L_{n-1}$ when $L_n$ contains multiple double points}
    \label{fig: two double}
    \end{center}
    \end{figure}
Lastly, instead of starting by defining $\phi(\mu_1)=g_1$ and $\phi(\mu_2)=g_2$, we instead define $\phi(\mu_1)=\phi(\mu_2)=g_1$. Since $\phi(\mu_n)$ can (and must) be chosen to generate a distinct subgroup from $\phi(\mu_1)$, $\phi$ will be surjective since $\phi(\mu_1)$ and $\phi(\mu_n)$ are linearly independent. Moreover, the anti-admissibility condition will be satisfied at the intersection $q=L_1\cap L_2$ by the eventual choice of $\phi(\mu_{n-1})$. Finally, from $\phi(\mu_1)=\phi(\mu_2)$ we must have $|C_q|\leq r-2$. It then follows that $|C_{n-1}|\leq n-3$.

There are three remaining cases for consideration. Either the arrangement contains no double points in addition to no lines being parallel, there are exactly two parallel lines or there is at least one double point, but no line contains more than one (double point). We will show the case when all lines intersect each other and all multiplicities are at least 3. The other two cases use similar reasoning to find the lines and modify $\phi$, with the exception that they start with a line that could take the place of $L_{n-1}$.

We may assume that all the intersection points of $\LA$ are triple points. Indeed, if there is a point with multiplicity $r\geq 5$, we set $r-2$ of the lines to be $L_1,\dots,L_{r-2}$ and $\phi(\mu_{1})=\dots=\phi(\mu_{r-2})=g_1$. Then the remaining two lines passing through this point can be taken to be $L_{n-1}$ and $L_{n}$. If there is a 4-fold point, let $L_n$ be a line on which this point is contained. Since $n\geq 5$, there is at least one other triple or 4-fold point on $L_n$. We label the lines as in Figure~\ref{fig: 4 fold} below (shown for the triple point case), and define $\phi(\mu_1)=\phi(\mu_2)=g_1$ and $\phi(\mu_3)=\phi(\mu_4)=g_2$ (if $L_3$ and $L_4$ intersect at a 4-fold point, the same assignments can be used). 

\begin{figure}[h]
    \begin{center}
    \includegraphics[scale=.5]{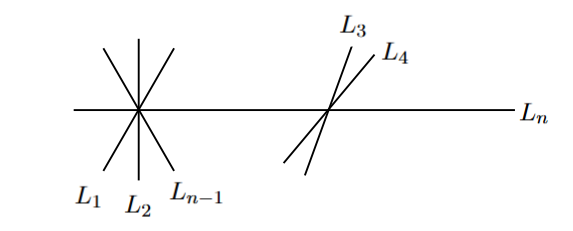}
    \caption{Labels when there is a quadruple point}
    \label{fig: 4 fold}
    \end{center}
    \end{figure}

Since $L_1\cap L_2\cap L_{n-1}$ is a triple point in $\LA_{n-1}$, and $\phi(\mu_1)=\phi(\mu_2)$, $|C_{n-1}|\leq n-3$ as required. The presence of multiple of the same meridian images also implies $|C_{n-1}|\leq n-3$. We may now assume $\LA$ is an arrangement with only triple points.

For this final case, let $L_n$ be any line in this arrangement. Since $n\geq 5$, it contains at least two triple points, and we will let the four other lines passing through these triple points be labeled as $L_1,\dots,L_4$ so that two of the intersection points on $L_n$ are $L_1\cap L_2\cap L_n$ and $L_3\cap L_4 \cap L_n$. Then since $L_1$ and $L_3$ must intersect at a triple point, there must be a another line in $\LA$ passing through $L_1\cap L_3$. We will index this line to be $L_{n-1}$ (Figure~\ref{fig: triple only}).

\begin{figure}[h]
    \begin{center}
    \includegraphics[scale=.5]{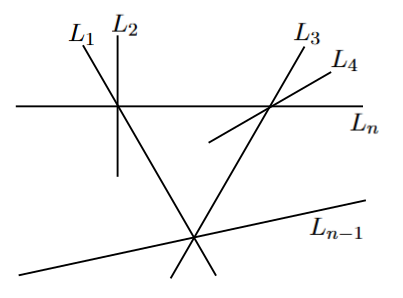}
    \caption{Labeling six of the lines in $\LA$}
    \label{fig: triple only}
    \end{center}
    \end{figure}

Finally, we define $\phi(\mu_1)=\dots=\phi(\mu_4)=g_1$, and then proceed with the induction as above. In this case, $|C_{n-1}|\leq n-3$ follows from $\phi(\mu_1)=\phi(\mu_3)$ whereas $|C_{n}|\leq n-3$ follows from $\phi(\mu_1)=\dots=\phi(\mu_4)$. We note that this assignment of the first four meridians does not impact the linear independence argument at any of the intersection points since for any $1\leq i<j\leq 4$, there must be a line of index $k>4$ for which $L_i\cap L_j\cap L_k$ is a triple point. The method above for choosing $\phi(\mu_k)$ will guarantee the anti-admissibility condition at this intersection. 

All other cases for which we have not provided the explicit method for arranging the lines and defining $\phi$ are proven via a combination of the ideas above and so we omit showing these cases. Nonetheless, the improvement to $p\geq n-1$ follows.
\end{proof}

\begin{remark}\label{rem: p=n,n-1 cases}

    As demonstrated by Example~\ref{example: general position k=2}, the above result is optimal when searching for a lower bound which depends only on the number of lines in the arrangement. However, the proof also shows that there may be a lot of cases for which other combinatorial data can provide further generalizations (see Questions~\ref{question: reasonable bounds} below). Regardless, one of the primary consequences of Theorem~\ref{theorem: affine mod p} is the following.
\end{remark}

\begin{corollary}\label{cor: finite values}
    Let $\LA=\{L_1,\dots,L_n\}$ be a line arrangement in $\C^2$ and $2\leq k\leq n$. Then there are at most finitely many pairs $(k,p)$ for which $\LA$ is not $(\Z/p\Z)^k$-anti-admissible.
\end{corollary}

Corollary~\ref{cor: finite values} now provides us with the following definition, which will also allow us to ask various questions surrounding anti-admissibility. The specific naming is motivated by Section~\ref{section: branched covers}.

\begin{definition}
   Fix an arrangement $\LA=\{L_1,\dots,L_n\}$,  for $2\leq k\leq n$, let $p_{\A}(\LA,k)$ denote the minimum prime $p_k$ for which $\LA$ is $(\Z/p\Z)^k$-anti-admissible for all primes $p\geq p_k$. We will refer to the set $\{p_{\A}(\LA,2),\dots, p_{\A}(\LA,n)\}$ as the \emph{(affine) smooth spectrum} of $\LA$. 
\end{definition}

From Lemma~\ref{lemma: vertical extension affine}, we know that the function $p_{\A}(\LA,k)$ is non-increasing for fixed $\LA$ and that $p_{\A}(\LA,|\LA|)=2$ for all arrangements by Example~\ref{example: canonical}. The primary question to then ask is what other behaviour is demonstrated by the spectra. 

In principle, the only information needed to compute $p_{\A}(\LA,k)$ is that contained in the incidence matrix (or the Levi graph) corresponding to the arrangement. Lemma~\ref{lemma: triple points bad} also shows that knowledge of the number of intersection points with certain multiplicities can provide restriction on anti-admissibility (to an extent this is also demonstrated by Remark~\ref{rem: double points LI} regarding arrangements with a large amount of double points). We are then left to ask how much information can be obtained using only intersection multiplicities? Put another way, can the spectra be computed without knowing how the singular points are distributed among the lines of the arrangement?

\begin{question}\label{question: reasonable bounds}
    What are the relations between the set of multiplicities $\{t_1,\dots,t_n\}$ and the spectra $p_{\A}(\LA,k)$? More specifically, for what classes of arrangements does the set of multiplicities completely determine the spectrum? 
\end{question}

Two initial examples for which the spectra can be explicitly computed are provided below.

\begin{example}\label{ex: spectrum GP}
    The same reasoning from Example~\ref{example: general position k=2} can be used for when $\mathcal G_n$ is an arrangement of $n$ lines in general position and $2\leq k\leq n$ is arbitrary. Then counting order $p$ subgroups of $(\Z/p\Z)^k$, if $1+p+\dots+p^{k-1}<n$, then $\mathcal G_n$ will not be $(\Z/p\Z)^k$-anti-admissible. On the other hand, if $p$ is any prime for which $1+p+\dots+p^{k-1}\geq n$, then any assignment of the meridians $\mu_i$ to generators of distinct (order $p$) subgroups will define a homomorphism that satisfies the anti-admissibility condition at all intersections. Furthermore, since $k\leq n$, by sending the first $k$ meridians to the standard basis elements, we can also guarantee surjectivity. Using this, the smooth spectrum of an arrangement in general position can be computed completely. Most importantly, this shows that for any $k$ there are choices for $\LA$ for which $p_{\A}(\LA,k)$ can be made arbitrarily large.
\end{example}

\begin{example}\label{ex: spectrum NP}
    Let $\LA=\{L_1,\dots,L_n\}$ be a near pencil arrangement, so that $L_n$ is a generic line and the sub-arrangement $\LA_{n-1}$ is a pencil. By Lemma~\ref{lemma: triple points bad}, if $n$ is even, then $p_{\A}(\LA,2)\neq 2$. If fact, one can check that this restriction coming from the parity of the multiplicity for the (unique) intersection point of $\LA_{n-1}$ is the only obstruction to anti-admissibility (cf.~\cite{harristhesis}).
    Hence the smooth spectrum of $\LA$ is either $\{2,\dots,2\}$ or $\{3,2,\dots,2\}$ in the odd and even cases, respectively. 
\end{example}

When one considers various examples of surjections satisfying anti-admissibility conditions appearing in the literature (for instance, \cite{akhmedovsakalliballquotients}, \cite{harpar24} or \cite{kul04}), it is often the case that the required primes are small when compared to the number of lines in the arrangement. Consequently, since the bounds suggested by Theorem~\ref{theorem: affine mod p} are closely related to those attained by generic arrangements as in Example~\ref{ex: spectrum GP}, we expect the following to be true when we consider larger values of $k$. 

\begin{conjecture}\label{conj: affine vs general}
    If $\LA$ is an arrangement and $\mathcal{G}_n$ is the arrangement of $n=|\LA|$ many lines in general position, then $p_{\A}(\LA,k)\leq p_{\A}(\mathcal G_n,k)$ for all $2\leq k\leq n$.
\end{conjecture}

By continuing to consider generic arrangements, it is clear that when $k$ is close to $n$, we have $p_{\A}(\mathcal G_n,k)=2$. Conjecture~\ref{conj: affine vs general} then suggests another interesting question.

\begin{question}
    For each $n>0$, what is the least $k\leq n$ for which $p_{\A}(\LA,k)=2$ for all arrangements containing $n$ lines? Namely, is it $\lceil\log_2(n+1)\rceil$?
\end{question}

\begin{remark}\label{rem: naive lower bounds affine}
    A generalization of Example~\ref{ex: spectrum GP} allows us to notice that if $\LA$ is an arrangement that contains $l$ many generic lines, then the sub-arrangement of these $l$ lines is clearly itself a generic arrangement. It then follows that
    $$p_{\A}(\mathcal G_l,k)\leq p_{\A}(\LA,k)$$
    for all $2\leq k\leq l$:
    Indeed, if $p\geq p_{\A}(\LA,k)$ and $\phi:H_1(\C^2\setminus\LA)\rightarrow(\Z/p\Z)^k$ is anti-admissible, then while we cannot conclude the induced homomorphism $\phi':H_1(\C^2\setminus\mathcal G_l)\rightarrow(\Z/p\Z)^k$ is surjective, we still know it must satisfy the anti-admissibility condition at all intersections. Since $\mathcal G_l$ is a generic arrangement, this means $1+p+\dots+p^{k-1}\geq l$. Then by Example~\ref{ex: spectrum GP} we can replace $\phi'$ by a new anti-admissible surjection provided $k\leq l$. 

    More generally, If $\LA'\subset \LA$ is any sub-arrangement for which the multiplicity of each point in $\textrm{Sing}(\LA')$ (assumed non-empty) is equal to its respective multiplicity in $\LA$, then 
    $$p_{\A}( \LA',2)\leq p_{\A}(\LA,2).$$
    These observations can be further applied to arrangements that admit certain decompositions.
\end{remark}

\begin{proposition}\label{prop: affine sum of arrangements}
    Suppose an arrangement can be decomposed as the union of $n$ (non-trivial) arrangements in $\C^2$ ($\LA=\LA_1\cup\dots\cup\LA_n,$ with $ |\LA_i|\geq 2$) so that any two lines $L_i\in \LA_i$ and $L_j\in\LA_j$ with $i\neq j$, intersect at a double point. Then 
    $$p_{\A}( \LA,2)\geq \max\{p_{\A}( \LA_1,2),\dots, p_{\A}( \LA_n,2),n-1\}.$$
\end{proposition}

\begin{proof}
    We notice that $\mathcal G_n$ appears as a sub-arrangement of $\LA$ by picking a single line from each $\LA_i$. The result then follows from applications of the inequalities in the above remark.
\end{proof}

\section{Anti-admissibility for projective arrangements}\label{sec: projective admissibility}

Now we consider the setting of arrangements in $\cptwo$. The main additions in this case comes from the added relation between meridian classes in homology as well as the fact that all lines must now intersect. Of these, we will see multiple times that this condition on homology provides interesting differences compared to the above definitions and results (for instance, Example~\ref{ex: differing spectra ceva}). Also recall that we assume all projective arrangements to not be a pencil.

We start by considering the results above with true analogous statements before later showcasing some of the finer differences between these cases.

\begin{proposition}\label{prop: projective integers}
    Let $\LA=\{L_1,\dots,L_n\}$ be an arrangement in $\cptwo$. Then $\LA$ is $\Z^k$-anti-admissible for $2\leq k\leq n-1$.
\end{proposition}

\begin{proof}
    We start by modifying the construction in Proposition~\ref{prop: affine integers} by adding (at most finitely many) additional values to the sets $A_{i+1}$. Namely they are the integer solutions to the equations defined as follows:
    For any (possibly non-empty) subset $I\subset \{1,\dots,i\}$ consider the equations
    $$\frac{1}{(r-1)-(n-1)}\left(-\sum_{k\in I} q_k\right)=q$$
    and 
    $$\frac{1}{(r-1)-(n-1)}\left(-q-\sum_{k\in I\setminus \{l\}}q_k\right)=q_l.$$
    where $l$ and $r$ vary over $1,\dots,i$ and $2,\dots,n-1$, respectively  (note in principle, we could refine $r$ to be based on $|I|$, but this is unnecessary as we only care about checking finitely many equations).
    We then choose $q_{i+1}$ (and consequently $\phi(\mu_{i+1})$) in a similar manner to before except for two exceptional cases. 
    
    First, when $i+1=n-1$, we consider two additional sets of equations and further add their solutions to $A_{n-1}$. Namely, for $I\subset \{1,\dots, n-2\}$, the equations
    $$\frac{1}{-(n-1)}\left(-q-\sum_{j=1}^{n-2}q_j\right)=\frac{1}{(r-1)-(n-1)}\left(-\sum_{k\in I}q_k\right)$$
    and
    $$\frac{1}{-(n-1)}\left(-q-\sum_{j=1}^{n-2}q_j\right)=\frac{1}{(r-1)-(n-1)}\left(-q-\sum_{k\in I}q_k\right)$$
    which we again allow for $r=2,\dots,n-1$. We may then choose $q_{n-1}$ appropriately.    
    
    When $i+1=n$, we instead simply define 
    $$\phi(\mu_{n})=(-(n-1),-\sum_{i=1}^{n-1}q_i).$$
    With this change, $\sum_{i=1}^n\phi(\mu_i)=(0,0)$ and so $\phi$ is well defined. Moreover, it is clear the anti-admissibility condition is satisfied at any intersection point that is not on $L_n$ by the exact same argument as before. 

    Therefore, we consider a point $p=L_{i_1}\cap\dots\cap L_{i_{r}}$ with $i_r=n$. Letting $q_i$ be denoted by $q_{L_i}$, we have
    $$\sum_{j=1}^r\phi(\mu_{i_j})=((r-1)-(n-1),-\sum_{L\in \LA,\ p\not\in L}q_L)$$
    and note that neither coordinate is zero since $\LA$ is not a pencil. If for some $i_l$, $l\neq n$, we didn't have linear independence, then there are $a_l,b_l\in \Z$ with 
    $$a_l(1,q_{i_l})=b_l((r-1)-(n-1),-\sum_{L\in \LA,\ p\not\in L}q_L)$$
    which implies 
    $$q_{i_l}=\frac{1}{(r-1)-(n-1)}\left(-\sum_{L\in \LA,\ p\not\in L}q_L\right).$$
    However, if $m$ is the largest index such that $p\not\in L_m$, then 
    $$\sum_{L\in \LA,\ p\not\in L}q_L=\sum_{i\in I\subset \{1,\dots,m\}}q_i=q_m+\sum_{i\in I\cap \{1,\dots,m-1\}}q_i$$
    where $I$ is the set of indices of lines $L\in\LA,\ p\not\in L$ (recall by definition $i_l\not\in I)$. It then follows that $q_m\in A_m$, a contradiction.
    
    Finally suppose 
    $$a_n(-(n-1),-\sum_{i=1}^{n-1}q_i)=b_n((r-1)-(n-1),-\sum_{L\in \LA,\ p\not\in L}q_L)$$
    so that
    $$\frac{1}{-(n-1)}\left(-\sum_{i=1}^{n-1}q_i\right)=\frac{1}{(r-1)-(n-1)}\left(-\sum_{L\in \LA,\ p\not\in L}q_L\right).$$
    Then regardless of whether or not $p\in L_{n-1}$, we would find that $q_{n-1}\in A_{n-1}$, another contradiction. Consequently, the linear independence criterion is satisfied for all meridians of lines passing through $p$. And so $\LA$ is $\Z^2$-anti-admissible.

    We may also extend this homomorphism to have image $\Z^k$ for $3\leq k\leq n-1$ by a similar shift as given in the affine case, albeit with 
    $$\phi(\mu_n)=(-(n-1)+k-2,-\sum_{i=1}^{n-1-k+2}q_i,-1,\dots,-1).$$
    One can verify these extensions still satisfy anti-admissibility.
\end{proof}

\begin{lemma}\label{lemma: vertical extension projective}
    Suppose a line arrangement $\LA=\{L_1,\dots,L_n\}$ in $\cptwo$ is $(\Z/p\Z)^k$-anti-admissible for a fixed prime $p$. If\/ $k+1\leq n-1$, then $\LA$ is $(\Z/p\Z)^{k+1}$-anti-admissible.
\end{lemma}

\begin{proof}
    The proof proceeds in the same fashion as Lemma~\ref{lemma: vertical extension affine}, with the exception that we must define the extension $\phi'$ so that $\sum_{i=1}^n\phi'(\mu_i)=0.$ To accomplish this, we assign the meridians as 
    $$\phi'(\mu_i)=(\phi(\mu_i),\delta_{i,k+1}-\delta_{i,k+2}).$$
    Using this, we see that $\phi'$ satisfies this required relation since $\sum_{i=1}^n\phi(\mu_i)$ is assumed to be trivial in $(\Z/p\Z)^k$. Moreover, $\{\phi'(\mu_1),\dots,\phi'(\mu_{k+1})\}$  will still generate $(\Z/p\Z)^k\oplus\Z/p\Z$, implying the surjectivity of $\phi'$.
\end{proof}

\begin{remark}
    Lemma~\ref{lemma: vertical extension projective}, in addition to Lemma~\ref{lemma: vertical extension affine} in the previous section, demonstrate that enlarging the group typically makes it more likely to find anti-admissible surjections (as one possibly expects, and provided $k+1$ does not exceed the rank of $H_1(Y\setminus \LA)$). An interesting question is then to ask if one can enlarge the group by changing the prime and still find anti-admissible surjections.  
\end{remark}

\begin{question}\label{question: horizontal extention possible}
    Suppose $\LA$ is $(\Z/p\Z)^k$-anti-admissible. Is $\LA$ $(\Z/q\Z)^k$-anti-admissible for all primes $q>p$?
\end{question}

As most of the arguments in this paper utilize combinatorial arguments to guarantee anti-admissibility provided enough elements are in the desired group, effectively all bounds considered will provide the property described in Question~\ref{question: horizontal extention possible}. However, for small $p$ (at least when compared to the number of lines in an arrangement), we are neither aware of a counterexample to this question nor do we know of a method to extend a morphism with image $(\Z/p\Z)^k$ to one with image $(\Z/q\Z)^k$ (while retaining various linear independence conditions). Nonetheless, a positive answer to this question will have a meaningful impact on our ability to compute the smooth spectra of an arrangement, as we will then only need to find (possibly) a single anti-admissible surjection compared to one for each prime. 

Unfortunately, the argument structure in Theorem~\ref{theorem: affine mod p} does not immediately generalize in the projective case. Namely, the proof in the previous section cannot guarantee $\sum_ic_ig_{m_i}=0$ since both the $c_i$ and $g_{m_i}$ are chosen to explicitly result in various sums being non-trivial. Nonetheless, as we expect the extreme case to correspond to arrangements of general lines (as motivated by Conjecture~\ref{conj: affine vs general} and the fact that most known examples have comparatively small primes), we expect an analogous result to hold. This is further suggested by considering the following.

\begin{theorem}\label{theorem: projective 2 generic}
    Let $\LA=\{L_1,\dots,L_n\}$ be an arrangement in $\cptwo$ that contains at least two generic lines. If $p> n$, then $\LA$ is $(\Z/p\Z)^2$-anti-admissible. 
\end{theorem}

\begin{proof}
    Let $L_{n-1}$ and $L_n$ denote two generic lines of $\LA$. For $i=1,\dots,n-2$, define $\phi(\mu_i)$ in an equivalent manner as the first part of the proof of Theorem~\ref{theorem: affine mod p}. We note that since $i-1\leq n-3<(p+1)-3$, we can assume each $g_{m_i}$ is neither $(1,0)$ nor $(0,1)$ nor $(1,p-1)$ .
    Let $(a,b)=\sum_{i=1}^{n-2}\phi(\mu_i)$. 
    
    Now, if $a=b=0\pmod{p}$, since $(n-2)-1<p-1$, we may modify our choice of $c_{n-2}$. In doing so, we will change both $a$ and $b$ to be non-zero in $\Z/p\Z$. We then set $\phi(\mu_{n-1})=(-a,0)$ and $\phi(\mu_{n})=(0,-b)$ so that $\sum_{i=1}^{n}\phi(\mu_i)=(0,0)$ and $\phi$ is a well defined surjection. As $L_{n-1}$ and $L_n$ are generic lines, they only contain double points and so Example~\ref{example: general position k=2} implies the linear independence condition is met at each intersection since no other meridian of lines correspond to the subgroups $\langle (1,0)\rangle$ and $\langle (0,1)\rangle$.

    If exactly one of $a$ or $b$ is zero in $\Z/p\Z$ (without loss we suppose $b=0$), then we set $\phi(\mu_{n-1})=-ag_{p-1}=(-a,-a(p-1))$ and $\phi(\mu_n)=(0,a(p-1)).$ We again find that $\sum_{i=1}^{n}\phi(\mu_i)=(0,0)$ and $\phi(\mu_{n-1})$ and $\phi(\mu_n)$ generate subgroups distinct from all other meridian images.

\end{proof}

\begin{remark}\label{rem: freedom of coefficients}
    The proof above demonstrates that provided we have some freedom in choosing the coefficient $c_i$, we can hope to obtain a bound on the prime $p$. In fact, it can be shown that for arrangements where all lines are generic the minimum possible value of $p$ is achievable (as there are no restrictions on the $c_i$ coming from the linear independence conditions). That is, if $p\geq n-1$, then $\mathcal G_n$ is $(\Z/p\Z)^2$-anti-admissible.  
\end{remark}

Corollary~\ref{cor: finite values} is also still true in the projective case as a consequence of a version of Theorem~\ref{theorem: projective 2 generic} that has been weakened to apply to all arrangements. A proof of this result (stated below) can be found in the author's thesis \cite{harristhesis}, although stated in different language. 

\begin{theorem}\label{theorem: thesis bound}
{\rm (cf.~Theorem 6.4.1 \cite{harristhesis})}
    Let $\LA$ be a line arrangement in $\cptwo$. For sufficiently large $p$, $\LA$ is $(\Z/p\Z)^2$-anti-admissible.
\end{theorem}

This result above also allows the smooth spectrum of an arrangement to be well defined when considered in the projective setting.

\begin{definition}
   Fix an arrangement $\LA=\{L_1,\dots,L_n\}$ in $\cptwo$. For $2\leq k\leq n-1$, let $p_{\mathbb P}(\LA,k)$ denote the minimum prime $p_k$ for which $\LA$ is $(\Z/p\Z)^k$-anti-admissible for all primes $p\geq p_k$. We will refer to the set $\{p_{\mathbb P}(\LA,2),\dots, p_{\mathbb P}(\LA,n-1)\}$ as the \emph{(projective) smooth spectrum} of $\LA$. 
\end{definition}

Aside from simply asking Question~\ref{question: reasonable bounds} in this new context, or restating Conjecture~\ref{conj: affine vs general}, we will instead focus on the question of how the smooth spectra differ between the affine and projective case. To do this, we will start by recalling some concepts that will allow us to move between the two regimes.

For an affine arrangement $\LA\subset \C^2$, $\Proj(\LA)$ will denote the (unique) projectivization of $\LA$. On the other hand, if the arrangement is instead projective ($\LA\subset \cptwo$), we will let $\Aff(\LA)_H$ correspond to the affine arrangement $\LA\setminus (\LA\cap H)$ in $\cptwo\setminus H$, where $H$ is a hyperplane of $\cptwo$. We note that if $H\subset \cptwo$ is chosen \emph{generically}, then for suitable affine coordinates on $\cptwo\setminus H$, the resulting arrangement of curves $\Aff(\LA)_H$ will have the same incidence matrix as $\LA$. However, different choices for $H$ may in principle result in different arrangements (recalling the work of Rybnikov \cite{rybnikovfund}). Regardless, since the spectra as defined above are completely determined by the incidence matrix of $\LA$, we will simply write $p_{\A}(\Aff(\LA),k)$ to refer to $p_{\A}(\Aff(\LA)_H,k)$ for any generic $H$. Lastly, we highlight that independent of the choice of generic hyperplane, we will still have $\Proj(\Aff(\LA)_H)=\LA$ (and consequently $p_{\mathbb P}(\Proj(\Aff(\LA)),k)=p_{\mathbb P}(\LA,k)$) for arrangements in $\cptwo$. On the other hand, in general for arrangements in $\C^2$, $p_{\A}(\Aff(\Proj(\LA)),k)\neq p_{\A}(\LA,k)$, as $\Aff(\Proj(\LA))_H$ and $\LA$ need not have the same intersection matrix, even for generic $H$. 

\begin{theorem}\label{theorem: aff proj compared}
    Let $\LA=\{L_1,\dots,L_n\}\subset \cptwo$ be a projective arrangement. Then $p_{\A}({\rm Aff}(\LA),k)\leq p_{\mathbb P}(\LA,k)$ for all $2\leq k\leq n-1$.
\end{theorem}

\begin{proof}
    For $p\geq p_{\mathbb P}(\LA,k)$, let $\phi:H_1(\cptwo\setminus\LA;\Z)\rightarrow (\Z/p\Z)^k$ be an anti-admissible surjection. Pick any generic hyperplane of $\cptwo$ (so that $H$ intersects $\LA$ at only double points). Then we have a sequence of homomorphisms

    \begin{tikzcd}
    \hspace{-0.4cm}H_1((\C^2\setminus \Aff(\LA)_H) \arrow[r] \arrow[r, "\cong"] & H_1(\cptwo\setminus (\LA\cup H)) \arrow[r] \arrow[r, "\psi"] & H_1(\cptwo\setminus\LA) \arrow[r] \arrow[r, "\phi"] & (\Z/p\Z)^k,
    \end{tikzcd}
    with the first isomorphism given by sending meridians around lines in $\C^2$ to meridians around the respective lines in $\cptwo$ via  
    \begin{equation*}
        H_1(\C^2\setminus \Aff(\LA)_H)\cong H_1((\cptwo\setminus H)\setminus (\LA\setminus (\LA\cap H))) \cong H_1(\cptwo\setminus (\LA\cup H)).
    \end{equation*}
    Note here that under this isomorphism, we have $\sum{\mu_i}\mapsto -\mu_H$.
    Furthermore, the homomorphism $\psi$ sends $\mu_L\mapsto \mu_L$ for each $L\in \LA$ and $\mu_H\mapsto 0$. By considering $\phi'$ as the composition, we arrive at our desired anti-admissible surjection for $\Aff(\LA)_H$.
\end{proof}

In general the spectra of an arrangement will not agree between the affine and projective case (at least for small values of $k$). This fact is a result from the additional relation the image of meridians must satisfy in the projective case.

\begin{example}\label{ex: differing spectra ceva}
    Let $\LA$ be the Ceva(4) arrangement (cf.~\cite{dolgachev} or \cite{barthelhirzebruch}) consisting of 12 lines, along which the $t_3=16$ triple points and $t_4=3$ 4-fold points are distributed evenly in the sense that each line contains exactly four triple points and one 4-fold point. A schematic of this (non-real) arrangement is provided in Figure~\ref{fig: ceva} below. We note that \emph{admissibility} for other Ceva arrangements has been considered in the past in \cite{bartoloetal}.

    \begin{figure}[h]
    \begin{center}
    \includegraphics[scale=.65]{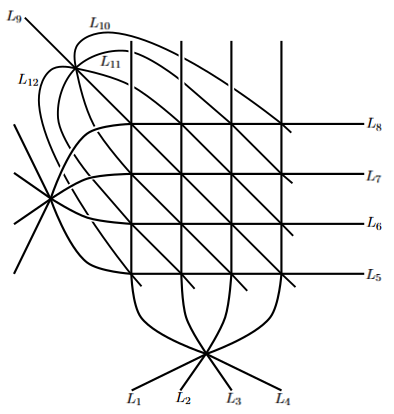}
    \caption{The Ceva(4) arrangement}
    \label{fig: ceva}
    \end{center}
    \end{figure}

    Now, lemma~\ref{lemma: triple points bad} implies $p_{\A}(\Aff(\LA),2)$ and $p_{\mathbb P}(\LA,2)$ are both at least 3. In fact, the assignment 
    \begin{equation*}
        \phi(\mu_i)=\begin{cases}
            (1,0) \quad i=1,3,5,7,10,12\\
            (0,1) \quad i=2,4,9,11\\
            (1,1) \quad i=6,8.
        \end{cases}
    \end{equation*}
    shows that the (affine) Ceva(4) is $(\Z/p\Z)^2$-anti-admissible for all $p\geq 3$

    Now, let $\phi:\{\mu_1,\dots,\mu_{12}\}\rightarrow (\Z/3\Z)^2\setminus\{(0,0)\}$ be any assignment for which the linear independence conditions of Definition~\ref{def: elementary anti-admissible} are satisfied for all intersection points. We will show that any such an assignment never defines a homomorphism.
    
    First, at any of the 4-fold points, the various linear independence conditions imply the images of the corresponding meridians will generate either 2 or 3 distinct order $3$ subgroups. We will begin with by distinguishing one of the 4-fold points (say  $p=p_{1,2,3,4}$), although up to symmetry, this choice does not meaningfully impact the arguments that follow. 
    
    We will also use $a,b,c,d\in (\Z/3\Z)^2$ to denote elements generating distinct subgroups. Then the possible images will correspond to an assignment of these elements (or their inverse) and will be covered on a case by case basis (up to permutation). Starting with there being three distinct subgroups generated by the images:

    \emph{Case 1.a:}  There are 3 distinct images $(\mu_1,\mu_2,\mu_3,\mu_4)\mapsto (a,a,b,c)$.

    Similar to Lemma~\ref{lemma: triple points p=3}, $-a,b$ and $c$ cannot each be pairwise linearly independent to $-a+b+c=2a+b+c$ since they generate distinct subgroups. It follows that this case can never occur if $\phi$ is assumed to be anti-admissible.

    \emph{Case 1.b:} There are 4 distinct images $(\mu_1,\mu_2,\mu_3,\mu_4)\mapsto (a,-a,b,c)$.

    Any line $L_i$ in the arrangement not passing through $p_{1,2,3,4}$ intersects each of the lines $L_1,\dots,L_4$ at a triple point. By considering the proof of Lemma~\ref{lemma: triple points p=3}, it follows that $\phi(\mu_i)\not\in\{-\phi(\mu_1),\dots,-\phi(\mu_4)\}=\{a,-a,-b,-c\}$. Further applications of Lemma~\ref{lemma: triple points p=3} then implies that $\phi(\mu_9)=\phi(\mu_5)=\phi(\mu_6)$ in addition to $\phi(\mu_{10})=\phi(\mu_5)$ and $\phi(\mu_{12})=\phi(\mu_6)$. Combining these conclusions, we find 
    $\sum_{i=9}^{12}\phi(\mu_i)=\phi(\mu_{11})$, a contradiction to the anti-admissibility condition at $p_{9,10,11,12}.$ 

    We may now suppose the meridians at any of the 4-fold points will correspond to exactly two distinct subgroups. In this setting, there are a few more options for $\phi$ (again, up to permutations). The argument for the (non-trivial) case \emph{2.a.2} is provided below, whereas the other cases follow immediately by computing $\sum_{i=1}^4\phi(\mu_i)$ and considering linear independence (cases \emph{2.a.1}, \emph{2.b.2} and \emph{2.c}) or by using a similar argument to case \emph{2.a.2} (case \emph{2.b.1}).   

    \emph{Case 2.a.1:} There are 2 distinct images $(\mu_1,\mu_2,\mu_3,\mu_4)\mapsto (a,a,a,b)$.

    \emph{Case 2.a.2:} There are 2 distinct images $(\mu_1,\mu_2,\mu_3,\mu_4)\mapsto (a,a,b,b)$.

    We now consider the images of the meridians $\mu_5,\dots,\mu_8$. These images must also generate exactly two distinct subgroups of $(\Z/3\Z)^2$. That is, up to the specific choices for these images, we also have a further two cases:
    
    First if $\{\mu_5,\mu_6,\mu_7,\mu_8\}\mapsto \{x,x,y,y\}$, then up to renaming we can assume $\phi(\mu_1)=a$ and $\phi(\mu_5)=x$. Therefore by Lemma~\ref{lemma: triple points p=3}, $\phi(\mu_9)\in \{a,x\}$. However, of the three remaining triple points on $L_9$, at least one must correspond to lines with meridians being sent to $b$ and $y$, which forces $\phi(\mu_9)\in \{b,y\}$, which is impossible. Consequently, $\{a,b,x,y\}$ cannot be a set of four distinct elements. 
    
    If this set consists of only two elements, so $\{a,b\}=\{x,y\}$. Then we must have  $\phi(\mu_i)\in \{a,b\}$ for each line $i\geq 9$. Moreover, there are lines $L_i,L_j$ with $1\leq i\leq 4$, $5\leq j\leq 8$ so that $\phi(\mu_i)=\phi(\mu_j)=a$. Let $L_k$ be the line passing through the triple point $p_{i,j,k}$. Then we must have $\phi(\mu_k)=b$. However, since $L_k$ must also intersect all other lines of index $\{1,\dots,8\}\setminus\{i,j\}$, at three distinct triple points, we see at least one such triple point must correspond to lines with meridians with images both $b$. At this point, linear independence is impossible.

    Alternatively, suppose $\{\mu_5,\mu_6,\mu_7,\mu_8\}\mapsto \{a,a,c,c\}$ for some distinct $c$ (which cannot be $-a$ nor $-b$). When we consider the points where the lines $L_j$, $5\leq j\leq 8$ intersect $L_1$, two of the intersection points correspond to lines with images $\phi(\mu_j)=a$. In this case the line $L_k$, $9\leq k\leq 12$ passing through either of these two points must have $\phi(\mu_k)\neq a$, and so $\phi(\mu_k)\in \{b,c\}$ when considering the other possible triple points on $L_k$. For the other two points on $L_1$, we find that the meridian $\mu_k$, $9\leq k\leq 12$ must satisfy $\phi(\mu_k)\in\{a,c\}$. However, when we further consider the other triple points on such an $L_k$, if $\phi(\mu_k)=c$, there is only one possible intersection with lines whose meridians map to $\{b,c\}$. That is, one of the intersections must correspond to the images $\{a,b\}$ which implies $\phi(\mu_k)\neq c$. Combining everything, as $p_{9,10,11,12}$ is a 4-fold point, we know that there can be only two subgroups generated by the corresponding images, and since case \emph{2.a.1} can be excluded, the images of $\{\mu_9,\mu_{10},\mu_{11},\mu_{12}\}$ are either $\{a,a,b,b\}$ or $\{a,a,c,c\}$. We then find in these respective cases that
    \begin{equation*}
        \sum_{i=1}^{12}\phi(\mu_i)=\begin{cases}
            6a+4b+2c\\
        6a+2b+4c
        \end{cases}
        =\begin{cases}
            b+2c\\
        2b+c\\
        \end{cases}
    \end{equation*}
    is non-trivial, as $b$ and $c$ generate distinct non-trivial subgroups. That is, any assignment satisfying all linear independence conditions will not generate a well defined homomorphism (for this case).

    If $\{\mu_5,\mu_6,\mu_7,\mu_8\}\mapsto \{x,x,-x,y\}$, then for any of the lines $L_i$, with $9\leq i\leq 12$, we cannot have $\phi(\mu_i)\in\{x,-x\}$ since $L_i$ intersects each of $L_5,\dots,L_8$ at a triple point. Furthermore, when we consider the intersection points on $L_1$, it follows that $\phi(\mu_i)=\phi(\mu_1)$ for at least three of the lines $L_9,\dots,L_{12}$, those corresponding to the intersections where $\phi(\mu_j)\in\{x,-x\}$, $5\leq j\leq 8$. Let $k\in\{9,\dots,12\}$ denote the line containing $L_1\cap L_j$, with $\phi(\mu_j)=y$ (and $j\in\{5,\dots,8\}$), then 
    $\sum_{i=9}^{12}\phi(\mu_i)=\phi(\mu_k)$, which is impossible.

    \emph{Case 2.b.1:} There are 3 distinct images $(\mu_1,\mu_2,\mu_3,\mu_4)\mapsto (a,a,-a,b)$.

    The argument in the previous case also shows that this is impossible (consider intersections on the line $L_5$ instead of $L_1$).

    \emph{Case 2.b.2:} There are 3 distinct images $(\mu_1,\mu_2,\mu_3,\mu_4)\mapsto (a,-a,b,b)$.

    \emph{Case 2.c:} There are 4 distinct images $(\mu_1,\mu_2,\mu_3,\mu_4)\mapsto (a,-a,b,-b)$.

    By considering all cases above, we see that any assignment of images either does not satisfy the linear independence conditions at each intersection point or it does not satisfy $\sum\phi(\mu_i)=0$ (and so it does not generate a homomorphism). Consequently, $p_{\mathbb{P}}(\LA,2)>3.$

    Now define $\phi:\{\mu_1,\dots,\mu_{12}\}\rightarrow(\Z/p\Z)^2\setminus\{(0,0)\}$ by
    \begin{equation*}
        \phi(\mu_i)=\begin{cases}
            (1,0) \quad &i=1,2,3,5,6,7\\
            (0,1) \quad &i=4,8,10,12\\
            (0,2) \quad &i=11\\
            (p-6,p-6)  &i=9
        \end{cases}
    \end{equation*}
    for primes $p\geq 5.$ By definition, $\phi$ generates a surjective homomorphism and since $p\neq 3$, each image is non-trivial. One can also check that at each intersection point, the various linear independence conditions are satisfied (this is most easily observed by considering the intersection of the corresponding generated subgroups). Consequently, $p_{\mathbb P}(\LA,2)=5\neq p_{\A}(\Aff(\LA),2)=3$. 
\end{example}

\begin{remark}\label{rem: algorithm scaling}
    Returning to Remark~\ref{rem: exhaustively check}, computer algorithms can be utilized to compute the smooth spectrum of a line arrangement. Unfortunately, the scaling for (albeit naive) algorithms explodes, as generic (sub)-arrangements will naturally enlarge the search space (Remark~\ref{rem: naive lower bounds affine}). That being said, Example~\ref{ex: differing spectra ceva} shows that arrangements may exhibit sufficient symmetry to reduce the required cases to levels that can be feasibly checked (at times by hand). Also somewhat surprisingly, while generic sub-arrangements will naturally require larger spectra, this is often counter-balanced by the fact that their images can be chosen to be any generator of a specified subgroup (recalling Remark~\ref{rem: freedom of coefficients}), making it easier to assign images which satisfied the requirements for anti-admissibility for non-generic lines. Consequently, we would be very interested in the discovery of efficient methods (if they exist) for computing the smooth spectra of line arrangements (especially in the projective case). This desire is further highlighted by the fact that the existence of anti-admissible surjections is closely tied to the construction of smooth 4-manifolds (Proposition~\ref{prop: main correspondence} below).  
\end{remark}

As demonstrated, the affine and projective spectra of an arrangement may not be equal. We are then left to ask to what extent they can differ. For instance, we know that to each (projective) arrangement $\LA$ there is a function $f(\LA,k)$ for which $p_{\mathbb P}(\LA,k)= \lceil p_{\A}(\Aff(\LA),k)+f(\LA,k)\rceil_P$ and so we can then consider what is the minimum data (i.e., number of lines, multiplicities, \dots) required to determine $f$. Specifically, is $f$ bounded a constant function? By considering examples of arrangements with a (relatively) small number of lines, we arrive at the (perhaps unrealistically optimistic) conjecture on the exact relation between the smooth spectra. 

\begin{conjecture}
   For a projective arrangement $\LA$, $p_{\mathbb P}(\LA,k)\leq \lceil p_{\A}({\rm Aff}(\LA),k)+1\rceil_P$. That is, $f$ can be defined to take values only in $\{0,1\}$. 
\end{conjecture}

Lastly, we return to Remark~\ref{rem: naive lower bounds affine} in an attempt to provide an analogous result and consider projective sub-arrangements. Now, not only will a surjection $\phi:H_1(\cptwo\setminus\LA)\rightarrow(\Z/p\Z)^k$ not necessarily define a surjection $\phi':H_1(\cptwo\setminus\LA')\rightarrow(\Z/p\Z)^k$, but since there is no requirement that the meridians around lines in $\LA'$ will have images with trivial sum, we cannot even use $\phi$ to define a related homomorphism with domain $H_1(\cptwo\setminus\LA')$. Nonetheless, we can at least combine Theorem~\ref{theorem: aff proj compared} with Remark~\ref{rem: naive lower bounds affine} to achieve bounds of the form
$$p_{\A}(\Aff( \LA'),2)\leq p_{\mathbb P}(\LA,2).$$

\section{Relation to branched covers}\label{section: branched covers}

The goal of this section is to justify a choice of generalization for anti-admissibility for any finite group. In doing so, this will also allow us to define what it means for a (finite) local system to be anti-admissible. 

First, recalling that unbranched $G$-covers of a space $Z$ are in correspondence with surjective homomorphisms $\pi_1(Z)\rightarrow G$, any of the surjections considered above will correspond to unbranched covers of the complement $Z=Y\setminus\LA$ by factoring through $H_1(Y\setminus\LA)$ via the Hurewicz map. Moreover, by a well known extension theorem of Grauert and Remmert (cf.~\cite{GraRem58}) it is possible to ``refill" the complement in the cover and obtain covers $X\rightarrow Y$ branched over the distinguished curve $\mathcal \LA$ (this is also true more generally for unions of curves in complex surfaces). As a consequence, in the context of local systems on line arrangements ($\phi:\pi_1(Y\setminus \LA)\rightarrow GL(n,\C)$), they are naturally closely related to branched covers with covering group $\phi(\pi_1(Y\setminus\LA))$. 

However, before we fully expand on this relation, we shall briefly review some facts surrounding branched covers (often specialized to the case where the branch set is a line arrangement). The primary reference for such covers  (at least in the abelian case) is Hironaka's monograph \cite{hiro93}, and in line with the previous two sections, after discussing some general concepts we will further specialize to the case when the covering group is $G=(\Z/p\Z)^k$ for primes $p$ and $k\geq 2$. While we will focus on the case when $Y=\cptwo$, Hironaka's results hold more generally when $Y$ is any irreducible and complex projective variety.

If $\rho:X\rightarrow Y$ is a branched covering (with branch set $B\subset Y$) and $\sigma:\widetilde{Y}\rightarrow Y$ is a birational morphism, then the \textit{pullback branched covering} $\widehat\rho:\widetilde{X}\rightarrow\widetilde{Y}$ (cf.~Definition~I.1.3 in \cite{hiro93})  is the minimal branched covering of $\widetilde{Y}$ making the diagram commute:
$$
\begin{tikzcd}
\widetilde{X} \arrow[r, "\widehat\sigma"] \arrow[d, "\widehat\rho"'] & X \arrow[d, "\rho"] \\
\widetilde{Y} \arrow[r, "\sigma"]                                    & Y                  
\end{tikzcd}
$$

We note that by definition, the branch locus of $\widehat\rho$ must be contained in $\sigma^{-1}(B)=\sigma^{-1}(\LA)$ while also containing the proper transform of each line $L_i\in\LA$ (denoted $\overline{L_i}$). Therefore, considering $\sigma$ as a composition of blow-ups, we see that the branch set of $\widehat\rho$ is a union of the proper transforms and some (possibly all, possibly none) of the exceptional spheres.

Furthermore, as an arbitrary cover of $Y$ branched over $\LA$ will not be smooth if $\textrm{Sing}(\LA)$ contains any point of multiplicity $r\geq 3,$ we are provided at least one natural choice of morphism to consider. We will denote by $\beta:\widetilde{Y}_{\LA}\rightarrow Y$, the (simultaneous) blow-up of $Y$ at all intersection points of $\LA$ with multiplicity $r\geq 3$. In this case, the meridian homology class (in $H_1(\widetilde{Y}_{\LA}\setminus \beta^{-1}(\LA))$) around an exceptional sphere $E=\beta^{-1}(p_{i_1,i_2,\dots,i_r})$ will be denoted by $\mu_E=\varepsilon_{i_1,i_2,\dots,i_r}$ and will satisfy the relation 
$$\varepsilon_{i_1,i_2,\dots,i_r}=\sum_{j=1}^r\beta^{-1}_*(\mu_{i_j}),$$
where $\beta_*$ is the isomorphism $H_1(\widetilde{Y}_{\LA}\setminus \beta^{-1}(\LA))\rightarrow H_1(Y\setminus \LA)$. It follows that the pullback branched covering in the abelian case (when $\rho$ is defined by $\phi:H_1(Y\setminus \LA)\rightarrow G$) will correspond to the homomorphism $\widetilde{\phi}=\phi\circ \beta_*$. Moreover, since $\widetilde{Y}_{\LA}\setminus \beta^{-1}(\LA)\cong Y\setminus \LA$ is a diffeomorphism, if we drop the abelian condition and consider a non-abelian covering defined by $\pi_1(Y\setminus \LA)\rightarrow G$, the pullback covering will still be induced by a composition $\widetilde{\phi}=\phi\circ \beta_*$, except now (by abusing notation) using the isomorphism $\beta_*:\pi_1(\widetilde{Y}_{\LA}\setminus \beta^{-1}(\LA))\rightarrow\pi_1(Y\setminus \LA)$.

\begin{remark}\label{rem: hirzebruch minimality}
    Even when considering $\sigma=\beta$, the pullback covering $\widetilde X$ may still not be smooth. In fact, even when it is smooth, it may also not be a minimal resolution of $X$. That said, there are some notable examples. For instance, in \cite{hir83}, Hirzebruch studied covers branched over arrangements $\LA\subset Y=\cptwo$ consisting of $n$ lines with covering group $G=(\Z/p\Z)^{n-1}$, defined by certain Kummer extensions of function fields. In doing so, they showed that the pullback branched covering (along $\beta$) is always smooth and is often minimal. Furthermore, explicit descriptions of the non-minimal covers were provided in these exceptional cases.  The relation between minimality and covers branched over line arrangements was also considered in \cite{kul04} in various examples. 
\end{remark}

With the understanding that pullback coverings may not be smooth, and that in general, a birational morphism $\sigma$ yielding a smooth cover must at least blow-up higher order intersection points (admit a decomposition $\sigma'\circ \beta$), we arrive at the following definition.

\begin{definition}\label{def: minimally smoothable}
   Let $G$ be any finite group. We say a $G$ cover $X\rightarrow Y$ branched over a line arrangement $\LA$ is \textit{minimally smoothable} if the pullback $G$-cover $\widetilde{X}\rightarrow\widetilde{Y}_{\LA}$ is smooth. Furthermore, if the branch locus of $\widehat\rho$ is exactly the total transform $\beta^{-1}(\LA)$, we will say the cover is \emph{totally} minimally smoothable.  
\end{definition}

It is known under what conditions finite abelian covers branched over arrangements of (smooth, transversely intersecting) curves in a complex projective variety will be smooth complex surfaces. This condition is summarized in the following lemma and will be referred to as the \textit{smoothness criterion}. 

\begin{lemma}\label{lemma: smoothness criterion}
{\rm (cf.~Proposition~I.4.1 in \cite{hiro93})}
An abelian branched covering $X\rightarrow Y$ defined by a surjection $\phi:H_1(Y\setminus B)\rightarrow G$ is smooth if and only if whenever two (distinct) curves  $C_1,C_2\subset B$ intersect, the subgroups $\langle\phi(\mu_{C_1})\rangle$ and $\langle\phi(\mu_{C_2})\rangle$ have trivial intersection (where as always, $\mu_{C_i}$ denotes the meridian homology class around $C_i$).
\end{lemma}

When $G=(\Z/p\Z)^k$ and $B=\LA\subset Y$ is a line arrangement, the smoothness criterion, when applied to a pullback branched covering, can be stated as the following result, a proof of which can be found in \cite{kul04}.

\begin{lemma}\label{lemma: kulikov smoothness}
{\rm (cf.~Lemmas~1.2 and 1.4 in \cite{kul04})}
Let $G=(\Z/p\Z)^k$ and $\widehat\rho: \widetilde{X}\to \widetilde{Y}_{\LA}$ be the pullback cover defined by a surjective homomorphism $\phi:H_1(Y\setminus\LA)\rightarrow G$. Then the complex surface $\widetilde{X}$\/ is smooth if and only if the following two conditions hold.  
\begin{itemize}
\item[(i)]  For each double point\/ $p_{i_1,i_2}=L_{i_1}\cup L_{i_2}$, we have  $\langle\phi(\mu_{i_1})\rangle\cap\langle\phi(\mu_{i_2})\rangle= \{0\}$. 

\item[(ii)] For each $r$-fold intersection point\/ 
$p_{i_1,i_2,\ldots,i_r}=L_{i_1}\cap L_{i_2} \cap \cdots \cap L_{i_r}$ with $r\geq 3$, we have  
$\langle\phi(\varepsilon_{i_1,i_2,\ldots,i_r})\rangle\cap\langle\phi(\mu_{i_j})\rangle = \{0\}$ for all
$j\in\{1,2,\dots, r\}$.  \hfill \qed
\end{itemize}
\end{lemma}

Importantly, we can now discuss the connection between $G$-anti-admissibility and smooth branched coverings.

\begin{proposition}\label{prop: main correspondence}
    Let $G=(\Z/p\Z)^k$ for some $k\geq 2$ and $p$ a prime. A line arrangement $\LA$ in $\cptwo$ is $G$-anti-admissible if and only if there exists a $G$ cover of $\cptwo$ branched over $\LA$ which is totally minimally smoothable.
\end{proposition}

\begin{proof}

    First, it was argued in \cite{hiro93} that a curve $C$ is in the branch locus of a cover defined by $\phi:H_1(Y\setminus B)\rightarrow G$ if and only if $\phi(\mu_C)\neq 0$ and so the branch locus of a pullback covering (induced by $\widetilde{\phi}=\phi\circ \beta_*$) will consist of $\overline{\LA}:=\overline{L_1}\cup\dots\cup\overline{L_n}$ and all exceptional spheres $E$ for which  
    $$\widetilde{\phi}(\varepsilon_{i_1,i_2,\dots,i_r})=\phi\circ\beta_*\left(\sum_{j=1}^r\beta^{-1}_*(\mu_{i_j})\right)=\sum_{j=1}^r\phi(\mu_{i_j})\neq 0.$$

    That is, a cover defined by $\phi$ is totally minimally smoothable if and only if it is minimally smoothable and $\sum_{j=1}^r\phi(\mu_{i_j})\neq 0$ for all points $p_{i_1,i_2,\dots,i_r}\in\textrm{Sing}(\LA)$ with multiplicity $r\geq 3$.

    Now, suppose $\LA$ is $G$-anti-admissible, so that a surjection $\phi:H_1(Y\setminus \LA)\rightarrow G$ satisfies Definition~\ref{def: elementary anti-admissible}. Property (1) implies that $\phi$ defines a cover of $\cptwo$ branched exactly over $\LA$. Furthermore, property (2), when combined with the observation in Example~\ref{example: general position k=2} implies the conditions in Lemma~\ref{lemma: kulikov smoothness} are satisfied. Therefore $\phi$ is minimally smoothable. Lastly, linear independence (2) also implies $\sum_{j=1}^r\phi(\mu_{i_j})\neq 0$ for all points $p_{i_1,i_2,\dots,i_r}\in\textrm{Sing}(\LA)$, so $\phi$ is totally minimally smoothable.
    
    Conversely, if $\phi:H_1(Y\setminus \LA)\rightarrow G$ defines a totally minimally smoothable cover, then $\phi$ is necessarily a surjection for which $\phi(\mu_i)\neq 0$ for all lines $L_i\in\LA$. It is straightforward to see that minimally smoothable implies either linear independence of pairs of the form
    $$\{\phi(\mu_{i_k}),\sum_{j\in \{i_1,i_2,\ldots,i_r\}}\phi(\mu_j)\}$$ 
    or that for some point $p_{i_1,i_2,\dots,i_r}\in\textrm{Sing}(\LA)$, 
    $$\sum_{j=1}^r\phi(\mu_{i_j})=0.$$
    The latter case contradicts the totality of the minimally smoothable cover, and so the result follows.
\end{proof}

By considering the above correspondence, in addition to the fact that Definition~\ref{def: minimally smoothable} is well defined for any finite group, we are provided with a natural definition for generalized $G$-anti-admissibility. Which will correspond to the relevant definitions above when $G=(\Z/p\Z)^k$.

\begin{definition}\label{def: G anti admissibility}
    Let $G$ be a finite group. A projective line arrangement $\LA\subset\cptwo$ is $G$\emph{-anti-admissible} if there exists a covering of $\cptwo$ that is branched exactly over $\LA$ that is totally minimally smoothable. 

    Moreover, if $\phi:\pi_1(\cptwo\setminus\LA)\rightarrow GL(k,\C)$ is a rank $k$ local system on $\LA$ with finite image, then $\phi$ is anti-admissible if the induced $\phi(\pi_1(Y\setminus\LA))$ cover is totally minimally smoothable. 
\end{definition}

\begin{remark}
    Recalling Definition~\ref{def: admissible}, we see that the only time a minimally smoothable cover can correspond to a $(\Z/p\Z)^k$-\emph{admissible} surjection is when no exceptional sphere is in the branch locus of the pullback covering. That is, the branch locus is $\overline\LA$. Furthermore, a surjection might yield a minimally smoothable covering yet be neither admissible nor anti-admissible. We feel encouraged to call any surjection coming from a minimally smoothable cover \emph{smoothly admissible}, and by extension, could consider smoothly admissible arrangements and local systems. While such surjections are not a focal point of this paper, we would nonetheless be interested in future research in this area as they are still a method for construction interesting smooth manifolds. For example, the applications of such surjections has been previously used to find examples of complex surfaces of general type with specified geometric genus \cite{kul04}.
\end{remark}

\begin{example}
    The coverings of Hirzebruch \cite{hir83} which were briefly discussed in Remark~\ref{rem: hirzebruch minimality} were in fact more generally covers with covering group $(\Z/m\Z)^{n-1}$ for any integer $m\geq 2$. Consequently, for any line arrangement (with $n$ lines), it is $(\Z/m\Z)^{n-1}$-anti-admissible for all integers $m\geq 2$.
\end{example}

\begin{example}
    In \S 4.4 of \cite{harristhesis}, $G$-anti-admissibility for non-cyclic groups whose order consists of multiple distinct prime factors was considered (more specifically, minimally smoothably covers were). Such groups included were $(\Z/p\Z)^k\oplus (\Z/q\Z)^{n-1-k}$ and $(\Z/p\Z)\oplus (\Z/q\Z)^{n-1-k}$ for $2\leq k\leq n-3$ to name a few, and some of the arrangements considered were generic and near pencil arrangements.
\end{example}

Since minimal smoothability for an abelian cover can be determined by only considering the incidence matrix of $\LA$ (a generalization of the observation for $(\Z/p\Z)^k$-anti-admissibility above), we can ask if this is still true in the non-abelian case. That is, we would also be interested in knowing whether there are arrangements with equivalent incidence matrices that can be distinguished by (non-abelian) anti-admissibility? As these are more closely tied to the fundamental group of the complement, we would expect the answer to be yes. That said, regardless of the answer, we will either possibly have an alternative method for distinguishing line arrangements, or it will imply that the smoothness of a large class of non-abelian branched covers is completely determined by the combinatorial data of the arrangement. Either of these cases can lead to advancements in at least one of these respective fields.

\begin{remark}
    We also highlight that Definition~\ref{def: elementary anti-admissible} and Definition~\ref{def: G anti admissibility} can also both be used when discussing arrangements of (non-hyperplane) curves in complex projective varieties. We hope to return to this idea in future work to investigate how the various generalizations of these definitions compare. In particular, we are interesting in whether or not there is an analogue for Proposition~\ref{prop: main correspondence} in this case. We highlight that research into how the existence of non-abelian branched coverings provides restrictions on the (non-hyperplane) curves in branching set has also been of interest for essentially just as long as the linear case (at least for dihedral groups \cite{tokunaga94},\cite{tokunaga00}), and we would be interested in learning if any of the ideas we have provided above may be useful to those currently considering such questions. 
\end{remark}

Another interesting question is to ask if there is any possible generalization of Theorem~\ref{theorem: thesis bound}. 

\begin{question}
    Given a line arrangement $\LA$, for what (finite) groups $G$ of rank $2\leq r\leq|\LA|$ is $\LA$ not $G$-anti-admissible? Is there any class of arrangements for which some finiteness result holds?
\end{question}

\begin{remark}
    We end by noting that general $G$-anti-admissibility has only been defined for the case when both the base space and the cover are compact 4 dimensional manifolds.  In principle however, a modified definition for the cases leading to non-compact covers is not impossible, although it may require more care. For instance, non-compact 4-manifolds are smoothable by the work of Freedman and Quinn \cite{FreQuin}, and so if a pullback covering is a topological 4-manifold, smoothability will follow, which may lead to complicated (and hopefully interesting) relations between the covering space $\widetilde{X}$, the induced covering map $\widehat\rho$, and the various algebraic structures on $\pi_1$ or $H_1$. These considerations become more involved when you further allow for the image of $\phi$ to be infinite. For these reasons, while we have only considered the compact cases above, we would be interested in returning to these intricacies in a future work.
\end{remark}

\bibliographystyle{abbrv}
\bibliography{references}

\end{document}